\documentclass[reqno]{amsart}
\usepackage{hyperref}
\usepackage{amsmath}
\usepackage{amssymb}
\usepackage{amsthm,upgreek,bm}
\usepackage{enumitem}

\usepackage[height=190mm,width=130mm]{geometry}
\usepackage[scr=boondoxo,scrscaled=1.1]{mathalfa} 

\theoremstyle{plain}
\newtheorem{theorem}{Theorem}[section]

\newtheorem{lemma}{Lemma}[section]
\newtheorem{corol}{Corollary}[theorem]

\theoremstyle{definition}
\newtheorem{definition}{Definition}
\newtheorem{remark}{\textup{Remark}}

\numberwithin{equation}{section}

\allowdisplaybreaks

\mathchardef\mhyphen="2D
 \newcommand{\abs}[1]{\left\vert #1\right\vert}%
\newcommand{\myp}[2][0cm]{\mathopen{}\left(#2\parbox[h][#1]{0cm}{}\right)}
\newcommand{\myb}[2][0cm]{\mathopen{}\left[#2\parbox[h][#1]{0cm}{}\right]}
 \newcommand{\set}[1]{\left\{ #1\right\}}%
\DeclareMathOperator{\RE}{\sf Re} %
\DeclareMathOperator{\Min}{Min} %
\DeclareMathOperator{\Max}{Max} %
\newcommand{\rs}[1]{\RE{\left\{ #1\right\}}} %
\newcommand{\A}{\mathcal{A}_p}%
\newcommand{\e}{\mathrm{e}}%
\renewcommand{\i}{\mathrm{i}}%
\newcommand{\q}{\mathfrak{q}}%
\newcommand{\p}{\mathfrak{p}}%
\newcommand{\ddash}[2]{
  {#1}\hspace{.5pt}\mbox{-}\nobreak{#2}%
}

\DeclareSymbolFont{symbolsC}{U}{pxsyc}{m}{n}
\DeclareMathSymbol{\coloneqq}{\mathrel}{symbolsC}{"42}
\allowdisplaybreaks

\begin{document}

\title[On the first-order differential subordination and
 $\cdots$ ]%
{On the  first-order differential subordination and
superordination results for $\bm{p}\,$-\hspace{0pt}valent  functions}
\author[V. S.  Masih, A. Ebadian, Sh. Najafzadeh]%
{Vali Soltani Masih, Ali Ebadian, Shahram Najafzadeh}

\newcommand{\acr}{\newline\indent}

\address{\textbf{Vali Soltani Masih}, Department of Mathematics\acr
                   Pyame Noor University(PNU)\acr
                   P.O. Box: 19395-3697\acr
                   Tehran, Iran}
\email{\url{masihvali@gmail.com; v\_soltani@pnu.ac.ir}}

\address{\textbf{Ali Ebadian}, Department of Mathematics\acr
                   Faculty of science\acr
                   Urmia university\acr
                   Urmia, Iran}
\email{\url{ebadian.ali@gmail.com}}

\address{\textbf{Shahram Najafzadeh}, Department of Mathematics\acr
                   Pyame Noor University(PNU)\acr
                   P.O. Box: 19395-3697\acr
                   Tehran, Iran}
\email{\url{najafzadeh1234@yahoo.ie}}


\begin{abstract}
In this paper, we obtain some application of first-order differential subordination, superordination and sandwich-type results involving operator for certain normalized \ddash{$p$}{valent} analytic functions. Further, properties of \ddash{$p$}{valent} functions such as; \ddash{$\uplambda$}{spirallike} and \ddash{$\uplambda$}{Robertson} of complex order are considered.
\end{abstract}
\maketitle
\section{Introduction}
Let $\mathcal{H}(\mathbb{U})$ denote the class of holomorphic functions in the open unit disc 
$\mathbb{U}\coloneqq\left\{z\in \mathbb{C}\colon |z|<1\right\}$ on the
complex plane $\mathbb{C}$, and let $\mathcal{H}[a,n]$ denote the subclass of the functions $\mathfrak{p}\in \mathcal{H}(\mathbb{U})$ of the form:
\[
\mathfrak{p}(z)=a+a_nz^n+\cdots;\qquad \myp{a\in \mathbb{C},\:n\in \mathbb{N}\coloneqq\set{1, 2, \ldots}}.
\]
   Let $\A$ denote the class of all \emph{$p$\nobreakdash-\hspace{0pt}valent}  functions $f\in \mathcal{H}$ of the following  form:
\begin{equation}\label{funh}
f(z)= z^p+\sum_{k=p+1}^{\infty}a_k z^k,
  \end{equation}
which are analytic  in the open unit disk $\mathbb{U}$. The class $\mathcal{A}_1$ denoted by $\mathcal{A}$. 

Let $g$ and $h$ be analytic in $\mathbb{U}$. We
say that the function $g$ is \emph{subordinate} to $h$, or the function $h$ is \emph{superordinate} to $g$, and express it by $g\prec h$ or conventionally by $g(z)\prec h(z)$ if $g = h\circ \upomega$
for some analytic map $\upomega\colon \mathbb{U} \rightarrow \mathbb{U}$ with $\upomega(0)=0$. When $h$ is univalent, the condition $g\prec h$ is
equivalent to $g(\mathbb{U})\subset h(\mathbb{U})$ and $g(0)=h(0)$.

 For some non-zero complex numbers $b$ and real $\uplambda$; $\myp{\abs{\uplambda}<\frac{\uppi}{2}}$, we define  classes $\mathcal{S}^{\uplambda}_p(\upalpha,b)$ and $ \mathcal{K}^{\uplambda}_p(\upalpha,b)$
 as follows: 
\begin{align*}
\mathcal{S}^{\uplambda}_p(\upalpha,b)&\coloneqq\set{f\in 
\A\colon\:\mathsf{Re}\left(\frac{1}{b\cos\uplambda}\myb{\e^{\i\uplambda}
            \frac{zf'(z)}{pf(z)}-\myp{1-b}\cos\uplambda-\i\sin\uplambda}\right)
        >\upalpha},
\intertext{and}
 \mathcal{K}^{\uplambda}_p(\upalpha,b)&\coloneqq\biggl\{f\in 
 \A\colon\:\mathsf{Re}\biggl(\frac{1}{b\cos\uplambda}    \biggl[\frac{\e^{\i\uplambda}}{p}
         \myp{1+\frac{zf''(z)}{f'(z)}}-
             \myp{1-b}\cos\uplambda-\i\sin\uplambda\biggr]\biggr)>\upalpha\biggr\}.
\end{align*}
 For a function $f$ belonging to the class 
{\mbox{\small$\mathcal{S}^{\uplambda}_p(\upalpha,b)$}}, we say that $f$ is 
\emph{multivalent $\uplambda\mbox{-}$spirallike of 
complex order $b$ and type $\upalpha$}; $\myp{0\leq \upalpha<1}$ in 
$\mathbb{U}$. Also for a function  $f$ belonging to the class 
{\mbox{\small$\mathcal{K}^{\uplambda}_p(\upalpha,b)$}}, we say that $f$ is 
\emph{multivalent $\uplambda\mbox{-}$Robertson of 
complex order $b$ and type $\upalpha$}; $\myp{0\leq \upalpha<1}$ in 
$\mathbb{U}$. This classes for $\upalpha=0$ were introduced and studied  by Ai-Oboudi and Haidan \cite{AH}.

 In particular  for $p=b=1$,  we denote
 \[
 \mathcal{S}^{\uplambda}(\upalpha)\coloneqq \mathcal{S}^{\uplambda}_1(\upalpha,1),
\]
is the class of \emph{$\uplambda$\nobreakdash-\hspace{0pt}spirallike functions of order $\upalpha$ with $0\leq \upalpha<1$} and
\[
\mathcal{K}^{\uplambda}(\upalpha)\coloneqq\mathcal{K}^{\uplambda}_1(\upalpha,1),
\]
is the class of \emph{$\uplambda$\nobreakdash-\hspace{0pt}Robertson functions of order $\upalpha$ with $0\leq \upalpha<1$}.

 Let $\upeta$ and $\upmu$ be complex numbers  not both equal to zero and $f\in \A$ given by \eqref{funh}. Define the differential operator $\mathscr{F}_{p}^{\upeta,\upmu}\colon \A \longrightarrow \mathcal{H}[1,1]$ as follows:
\begin{equation}\label{diff}
\mathscr{F}_{p}^{\upeta,\upmu}[f](z)\coloneqq\myb{\frac{f'(z)}{pz^{p-1}}}^{\upeta} \myb{\frac{z^p}{f(z)}}^{\upmu}=1+\myp{\upeta-\upmu+\frac{\upeta}{p}}a_{p+1}z+ \cdots; \quad \myp{z\in \mathbb{U}}, 
\end{equation}
with $\mathscr{F}_{p}^{\upeta,\upmu}[f](z)\Bigr|_{z=0}=1$. Here, all powers are mean as principal values (see \cite{EMN}). \smallskip

\section{Definitions and Preliminaries}
In order to achieve our aim in this section, we recall some definitions and preliminary results from the theory of differential subordination and superordination.
\begin{definition}[\cite{MM2,MM3}]

Let $\uppsi\colon\mathbb{C}^2\times \mathbb{U}\rightarrow \mathbb{C}$ and the function $h(z)$ be
univalent in $\mathbb{U}$. If the function $\p(z)$ is analytic in $\mathbb{U}$ and satisfies the following first-order differential subordination
\begin{equation}\label{first-order}
\uppsi\myp{\mathfrak{p}(z),z\mathfrak{p}'(z);z}\prec h(z); \qquad \myp{z\in \mathbb{U}},
\end{equation}
then $\p(z)$ is called a \emph{solution} of the differential subordination.

  A function $\mathfrak{q}\in\mathcal{H}$ is said to be a \emph{dominant} of the differential subordination \eqref{first-order} if $\mathfrak{p}\prec \mathfrak{q}$ for all $\mathfrak{p}$ satisfying \eqref{first-order}.
An univalent dominant that satisfies $\mathfrak{\tilde{q}}\prec \mathfrak{q}$ for all dominants $\mathfrak{q}$ of \eqref{first-order}, is said to be \emph{best dominant} of the differential subordination.\smallskip
\end{definition}
\begin{definition}[\cite{MM4}]
Let $\upvarphi\colon \mathbb{C}^2\times \mathbb{U}\rightarrow \mathbb{C}$ and the function $h(z)$ be univalent in $\mathbb{U}$. If the
function $\mathfrak{p}(z)$ and $\upvarphi\myp{\p(z), z\p'(z); z}$ are univalent in $\mathbb{U}$ and satisfies the following first-order differential superordination
\begin{equation}\label{first-order-super}
h(z)\prec \upvarphi\myp{\p(z), z\p'(z); z}; \qquad \myp{z\in \mathbb{U}},
\end{equation}
then $h(z)$ is called a \emph{solution} of the differential superordination.

An analytic function $\q\in \mathcal{H}$ is called a \emph{subordinant} of the solution of the differential superordination \eqref{first-order-super}, or more simply a subordinant if $\q\prec \p$ for all the functions $\p$ satisfying \eqref{first-order-super}. An univalent subordinant that satisfies $\q\prec \tilde{\q}$ for all of the subordinants $\q$ of
\eqref{first-order-super}, is said to be the \emph{best subordinant}.

  Miller and Mocanu \cite{MM4} obtained sufficient condition on the functions $\p$ and $\q$ for which the following implication holds:
\[
h(z)\prec \upvarphi\myp{\p(z), z\p'(z); z} \Longrightarrow \q(z)\prec \p(z).
\]
\end{definition}
Using these results, in \cite{bul2} were obtained sufficient conditions for 
certain normalized analytic function f to satisfy
\[
\q_1(z) \prec \frac{zf'(z)}{f(z)}\prec \q_2(z),
\]
where $\q_1(z)$ and $\q_2(z)$ are given univalent normalized function in $\mathbb{U}$.\smallskip
\begin{definition}[cf. Miller and Mocanu{\cite[Definition 2.2b, 
p.21]{MM1}}]
Denote by $\mathcal{Q}$, the
set of all functions $f(z)$ that are analytic and injective on
$\overline{\mathbb{U}}\setminus \mathbf{E}(f)$, where
\[
 \mathbf{E}(f)\coloneqq\left\{\upzeta\colon \:\: \upzeta \in\partial \mathbb{U}\quad \text{and}\quad \lim_{z\rightarrow \upzeta}
f(z)=\infty \right\},
\]
  and are such that $\Min\abs{f'(\upzeta)}=\rho >0$ for
$\upzeta\in\partial\mathbb{U}\setminus \mathbf{E}(f)$.
\end{definition}
\begin{lemma}[cf. Miller and Mocanu{\cite[Theorem 3.4h, p.132]{MM1}}]\label{lem1}
Let $\q$ be univalent in $\mathbb{U}$, and let $\upvarphi$ and $\uptheta$ be analytic in a domain $\Omega$ containing $\q\myp{\mathbb{U}}$, with $\upvarphi(w)\neq 0$ when $w\in \q\myp{\mathbb{U}}$. Set $Q(z)\coloneqq z\q'(z)\upvarphi(\q(z))$; $h(z)\coloneqq\uptheta(\q(z))+Q(z)$ and suppose that
\begin{itemize}
\item[(i)]  
$Q(z)$ is starlike function in $\mathbb{U}$,
 \item[(ii)]
$\rs{\frac{zh'(z)}{Q(z)}}=\rs{\frac{\uptheta'\myp{\q(z)}}{\upvarphi\myp{\q(z)}}+\frac{zQ'(z)}{Q(z)}}>0$ for  $z\in \mathbb{U}.$
\end{itemize}
If $\mathfrak{p}(z)$ is analytic in $\mathbb{U}$, with $\mathfrak{p}(0)=\q(0),\;
\mathfrak{p}\myp{\mathbb{U}}\subset \Omega$ and
\begin{equation} \label{eq8}
    \uptheta(\mathfrak{p}(z))+z\mathfrak{p}'(z)\upvarphi(\mathfrak{p}(z))\prec
    \uptheta(\q(z))+z \q'(z)\upvarphi(\q(z))=h(z); \qquad z\in \mathbb{U},
\end{equation}
then $\mathfrak{p}(z)\prec \q(z)$ and $\q$ is the best
dominant of Eq.~\eqref{eq8}.
\end{lemma}
\begin{lemma}[\cite{SRS}]\label{SRS}
Let $\q(z)$ be convex function in $\mathbb{U}$ and $\upgamma\in \mathbb{C}$ with $\rs{\upgamma}>0$. If $\mathfrak{p}(z)\in \mathcal{H}[\q(0),1]\cap \mathcal{Q}$ and  $ \p(z) +\upgamma z \p'(z)$ is univalent in $\mathbb{U}$, then
\begin{equation}\label{lem6}
 \q(z) +\upgamma z \q'(z) \prec  \p(z) +\upgamma z \p'(z)
\end{equation}
implies $\q(z)\prec \p(z)$ and $\q(z)$ is the best subordinant of Eq.~\eqref{lem6}.
\end{lemma}
\begin{lemma}[\cite{Roy}]\label{lem:roster}
The function 
\[
\q_{\uplambda}(z)\coloneqq\myp{1-z}^{\uplambda}\equiv \e^{\uplambda \log \myp{1-z}}=1-\uplambda z+\frac{\uplambda \myp{\uplambda-1}}{2}z^2-\frac{\uplambda\myp{\uplambda-1}\myp{\uplambda-2}}{6}z^3+\cdots
\]
for some  $\uplambda\in \mathbb{C}^{\ast}\coloneqq \mathbb{C}\setminus\set{0}, z\in \mathbb{U}$ is univalent in $\mathbb{U}$ if and only if $\uplambda$ is either in the closed disk $\abs{\uplambda+1}\leq 1$ or $\abs{\uplambda-1}\leq 1$.
\end{lemma}
\begin{lemma}\label{lemma2}
For the univalent functions
\begin{enumerate}[label=\textbf{(UF.\arabic*)}]
\item\label{UF1}
$\q(z)=\myp{1+Bz}^\uplambda$ with
\[
 -1\leq B\leq 1;\, B\neq0\quad \text{and}\quad\uplambda\in \mathbb{C}^{\ast} \quad \text{with}\quad \abs{\uplambda+1}\leq 1\quad \text{or}\quad \abs{\uplambda-1}\leq 1,
\]
\item\label{UF2}
 and
\[
\q(z)=\frac{1+Az}{1+Bz};\quad \myp{-1\leq B<A\leq 1, \, z\in \mathbb{U}},
\]
\end{enumerate}
we have
\begin{equation}\label{eqlemma21}
\rs{1+\frac{z\q''(z)}{\q'(z)}-\frac{z\q'(z)}{\q(z)}}>0; \qquad \myp{z\in \mathbb{U}}.
\end{equation}
\end{lemma}
\begin{proof}
\begin{enumerate}[label=\textbf{UF.\arabic*}]
\item
 From Lemma \ref{lem:roster}, the function $\q(z)=\myp{1+Bz}^\uplambda$ univalent in  $\myp{z\in 
 \mathbb{U}}$. A simple calculations shows that
 \[
 \rs{1+\frac{z\q''(z)}{\q'(z)}-\frac{z\q'(z)}{\q(z)}}=\rs{\frac{1}{1+Bz}}>\frac{1}{1+|B|}>0.
 \]
\item
Let $\q(z)=\myp{1+Az}/\myp{1+Bz}$; $\myp{-1\leq B<A\leq 1, \: z\in \mathbb{U}}$, then we have
\begin{equation*}
\rs{1+\frac{z\q''(z)}{\q'(z)}-\frac{z\q'(z)}{\q(z)}}=\rs{\frac{1-ABz^2}{\myp{1+Az}\myp{1+Bz}}}.
\end{equation*}
 The function 
 \[
 \p_{A,B}(z)=\frac{1-ABz^2}{\myp{1+Az}\myp{1+Bz}}; \qquad \myp{-1\leq B<A\leq 1},
 \] 
 dose not have any poles in $\overline{\mathbb{U}}$ and is analytic in $\mathbb{U}$. Then
 \[ 
 \Min\set{\rs{\p_{A,B}(z)}\colon\: |z|<1},
 \]
   attains its minimum value on the boundary $\set{z\in \mathbb{C}\colon |z|=1}$. If take $z=\e^{\i \theta}$ with $\theta\in (-\uppi,\uppi]$, then
 \begin{equation}\label{eqlemma22}
\rs{\frac{1-AB 
\e^{2\i\theta}}{\myp{1+A\e^{\i\theta}}\myp{1+B\e^{\i\theta}}}}=
\frac{\myp{1-AB}\myb{1+AB+\myp{A+B}\cos\theta}}{\abs{1+A\e^{\i\theta}}^2\abs{1+B\e^{\i\theta}}^2}.
 \end{equation}
 If $A+B\geq 0$, it follows that $1+AB+\myp{A+B}\cos\theta \geq \myp{1-A}\myp{1-B} \geq 0$,  and if
$A+B\leq 0$, it follows that $1+AB+\myp{A+B}\cos\theta \geq \myp{1+A}\myp{1+B}\geq 0$. Therefore, the minimum value of expression \eqref{eqlemma22} is  equal to $0$. \qedhere
   \end{enumerate}
\end{proof}
\begin{lemma}[\cite{BT}]\label{lemmaBT}
Let $\q$ be  function in $\mathbb{U}$ with $\q(0)\neq 0$. If $\q$ satisfy the condition \eqref{eqlemma21}, then for all $z\in \mathbb{U}$, $\q(z)\neq 0$.
\end{lemma}

\begin{lemma}\label{lemma4}
For the  function $\q(z)=\myp{1+Az}/\myp{1+Bz}$; $-1\leq B<A\leq 1, \, z\in \mathbb{U}$  the condition
\begin{equation}\label{eqlemma41}
\rs{1+\frac{z\q''(z)}{\q'(z)}} > \Max \set{0, -\RE\left(\upzeta\right)}; \qquad \myp{z\in \mathbb{U},\upzeta \in \mathbb{C}},
\end{equation}
equivalent to
$
\rs{\upzeta}\geq \frac{|B|-1}{|B|+1}.
$
\end{lemma}
\begin{proof}
The function $\upomega(z)=1+\frac{z\q''(z)}{\q'(z)}=\frac{1-Bz}{1+Bz}$; $\myp{-1\leq B<A\leq 1, \, B\neq 0}$, maps unit disk $\mathbb{U}$ onto the disk 
\begin{align*}
&\abs{\upomega(z)-\frac{1+B^2}{1-B^2}} <\frac{2\abs{B}}{1-B^2}; \qquad \myp{z\in \mathbb{U}},
\intertext{which implies that}
&\rs{\upomega(z)}>\frac{1-|B|}{1+|B|}; \quad \myp{z\in \mathbb{U}}.
\end{align*}
From \eqref{eqlemma41} we have
\[
\frac{1-|B|}{1+|B|}\geq \Max \set{0, -\RE\left(\upzeta\right)}
\]
and this is equivalent to  $\rs{\upzeta}\geq \myp{|B|-1}/\myp{|B|+1}$.
\end{proof}
\begin{lemma}\label{lemma3}
Let 
\[
\upomega(z)=\frac{u+vz}{1+Bz}; \qquad \myp{u,v\in \mathbb{C};\,\,\text{with}\,\, (u,v)\neq (0,0),\,\,-1<B<1,\, z\in \mathbb{U}}.
\] 
Suppose  that
$\rs{u-vB}\geq \abs{v-uB}$, then $\rs{\upomega(z)}>0$; $\myp{z\in \mathbb{U}}$.
\end{lemma}
\begin{proof}
The function $\upomega(z)=\frac{u+vz}{1+Bz}$ maps $\mathbb{U}$ onto the disk
\begin{align*}
&\abs{\upomega(z)-\frac{u-vB}{1-B^2}} <\frac{\abs{v-uB}}{1-B^2}; \qquad \myp{z\in \mathbb{U}},
\intertext{which implies that}
&\rs{\upomega(z)}>\frac{\rs{u-vB}-\abs{v-uB}}{1-B^2}\geq 0; \quad \myp{z\in \mathbb{U}}.\qedhere
\end{align*}
\end{proof}

 Some interesting results of differential subordination and superordination 
 were obtained recently (for example) Bulboac\u{a} 
 \cite{Bul,bul2,BBS}, Shammugam et al. 
 \cite{SRMS}, Zayed et al. \cite{ZMA}, Ebadian and Sok{\'o}{\l} 
 \cite{ES} and Aouf et al. \cite{AMZ}. 
 
 In this paper, we will derive several 
 subordination, superordination and sandwich results involving 
 the operator $\mathscr{F}_{p}^{\upeta,\upmu}$.
 
\section{Subordination Results}\label{sec:sub}
For convenience, let
\begin{align*}
\mathcal{A}_0 &\coloneqq \biggl\{f\in \A \quad\colon\quad \mathscr{F}_{p}^{\upeta,\upmu}[f](z)\Bigr|_{z=0}=1,\,\,  \upeta,\upmu\in \mathbb{C};\, (\upeta,\upmu)\neq(0,0)\biggr\}.\\
\mathbf{B}&\coloneqq \set{z\in \mathbb{C}\quad\colon \quad\abs{z+1}\leq 1\quad \text{or}\quad \abs{z-1}\leq 1}.
\end{align*}
We assume in the remainder of this paper that
 $\upsigma$ be complex number, $\upgamma\in \mathbb{C}^{\ast}$, $\upalpha, \uplambda$ are real numbers with $0\leq \upalpha<1, \, -\frac{\uppi}{2}<\uplambda<\frac{\uppi}{2}$, respectively, and all the powers are principal ones.
\begin{theorem}\label{th:a}
Let $\q$ be univalent in $\mathbb{U}$ with $\q(0)=1$, and $\q$ satisfy the condition \eqref{eqlemma21}. If the function $f \in \mathcal{A}_0$ with $\mathscr{F}_{p}^{\upeta,\upmu}[f](z)\neq 0$; $\myp{z\in \mathbb{U}}$  satisfies the following subordination condition:
\begin{equation}
1+\upgamma \myb{\upeta\myp{1-p+\frac{zf''(z)}{f'(z)}}+\upmu\myp{p-\frac{zf'(z)}{f(z)}}} \prec 1+\upgamma \frac{z\q'(z)}{\q(z)}; \qquad \myp{z\in \mathbb{U}},\label{th:aaaa}
\end{equation}
then
\[
\mathscr{F}_{p}^{\upeta,\upmu}[f](z)\prec \q(z); \qquad \myp{z\in \mathbb{U}},
\]
and $\q$ is the best dominant of Eq.~\eqref{th:aaaa}.
\end{theorem}
\begin{proof}
 If we choose $\uptheta(w)=1$ and $\upvarphi(w)=\frac{\upgamma}{w}$, then $\uptheta, \upvarphi \in \mathcal{H}(\Omega)$; $\myp{\Omega\coloneqq\mathbb{C}^{\ast}}$. The condition $\q(\mathbb{U})\subset \Omega$ from Lemma \ref{lem1} is equivalent to $\q(z)\neq 0$ for all $z\in \mathbb{U}$. For $w\in \q(\mathbb{U})$, we have $\upvarphi(w)\neq 0$. Define
 \begin{equation*}
 Q(z)\coloneqq z \q'(z)\upvarphi(\q(z))=\upgamma \frac{z\q'(z)}{\q(z)};\qquad \myp{z\in \mathbb{U}}.
 \end{equation*}
From Lemma \ref{lemmaBT}, $\q(z)\neq 0$ for all $z\in \mathbb{U}$, then $Q\in \mathcal{H}(\mathbb{U})$. Further, $\q$ is an univalent function, implies $\q'(z)\neq0$ for all $z \in\mathbb{U}$, $Q(0)=0$ and $Q'(0)=\upgamma \frac{\q'(0)}{\q(0)}\neq 0$, and 
 \begin{align*}
&\rs{\frac{zQ'(z)}{Q(z)}}=\rs{1+\frac{z\q''(z)}{\q'(z)}-\frac{z\q'(z)}{\q(z)}}>0; \qquad \myp{z\in \mathbb{U}},
\intertext{hence $Q$ is a starlike function in $\mathbb{U}$. Moreover, if}
 &h(z)\coloneqq\uptheta(\q(z))+Q(z)=1+\upgamma \frac{z\q'(z)}{\q(z)},
 \intertext{we also have}
& \rs{\frac{zh'(z)}{Q(z)}}=\rs{\frac{zQ'(z)}{Q(z)}}>0; \qquad \myp{z\in \mathbb{U}}.
 \end{align*}
For $f\in \mathcal{A}_0$, the function $\mathscr{F}_{p}^{\upeta,\upmu}[f](z)$ given by \eqref{diff}, we have $\mathscr{F}_{p}^{\upeta,\upmu}[f](\mathbb{U})\subset\Omega$ and the subordinations \eqref{eq8} and \eqref{th:aaaa} are equivalent, then  all the conditions of Lemma \ref{lem1} are satisfied and the function $\q$ is the best dominant of \eqref{th:aaaa}.
\end{proof}
Taking $\upeta=0$, $\upgamma=1$ and $\q(z)=\myp{1+Az}/\myp{1+Bz}$; $\myp{-1\leq A<B\leq 1, \: z\in \mathbb{U}}$ in Theorem \ref{th:a} and applying item \ref{UF2}, we get the following result:
\begin{corol}
Let $-1\leq A<B\leq 1$, $\upmu\neq 0$ and $f\in \A$ satisfy the conditions
\begin{align}
& \myb{\frac{z^p}{f(z)}}^{\upmu}\Biggr|_{z=0}=1  \quad \textrm{and}\quad  \frac{z^p}{f(z)}\neq 0; \qquad \myp{z\in \mathbb{U}}. \notag
\intertext{If the function $f$  satisfies the following subordination condition:}
&1+\upmu\myp{p-\frac{zf'(z)}{f(z)}} \prec 1+ \frac{\myp{A-B}z}{\myp{1+Az}\myp{1+Bz}}; \qquad \myp{z\in \mathbb{U}},\label{coro:th1:1}
\end{align}
then
\begin{equation*}
\myp{\frac{z^p}{f(z)}}^{\upmu}\prec \frac{1+Az}{1+Bz}; \qquad \myp{z\in \mathbb{U}},
\end{equation*}
and $\myp{1+Az}/\myp{1+Bz}$ is the best dominant of Eq.~\eqref{coro:th1:1}
\end{corol}
Taking $\upmu=0$, $\upgamma=1$  and $\q(z)=\myp{1+Az}/\myp{1+Bz}$; $\myp{-1\leq A<B\leq 1, \: z\in \mathbb{U}}$ in Theorem \ref{th:a} and applying item \ref{UF2}, we get the following result:
\begin{corol}
Let $-1\leq A<B\leq 1$, $\upeta\neq 0$ and $f\in \A$ satisfy the conditions
\begin{align}
& \myb{\frac{f'(z)}{pz^{p-1}}}^{\upeta}\Biggr|_{z=0}=1  \quad \textrm{and}\quad  \frac{f'(z)}{pz^{p-1}}\neq 0; \qquad \myp{z\in \mathbb{U}}. \notag
\intertext{If the function $f$  satisfies the following subordination condition:}
&1+\upeta\myb{1-p+\frac{zf''(z)}{f'(z)}}  \prec 1+ \frac{\myp{A-B}z}{\myp{1+Az}\myp{1+Bz}}; \qquad \myp{z\in \mathbb{U}},\label{coro:th1:2}
\end{align}
then
\[
\myb{\frac{f'(z)}{pz^{p-1}}}^{\upeta}\prec \frac{1+Az}{1+Bz}; \qquad \myp{z\in \mathbb{U}},\notag
\]
and $\myp{1+Az}/\myp{1+Bz}$ is the best dominant of \eqref{coro:th1:2}
\end{corol}
Taking  $\upgamma=\frac{\e^{\i \uplambda}}{pab\cos\uplambda}$, $\upmu=-a$, $\upeta=0$ and  $\q(z)=\myp{1-z}^{-2pab\myp{1-\upalpha}\e^{-\i \uplambda}\cos \uplambda}$ in Theorem \ref{th:a} and combining this together with item  \ref{UF1}, we obtain the following result:
\begin{corol}\label{corol:th1:3}
Let $f\in \mathcal{S}^{\uplambda}_p(\upalpha,b)$. Then
\begin{align}
\myb{\frac{f(z)}{z^p}}^{a} &\prec \frac{1}{\myp{1-z}^{2pab\myp{1-\upalpha}\e^{-\i \uplambda}\cos \uplambda}}; \qquad \myp{a\in \mathbb{C}^{\ast},\:z\in \mathbb{U}}.\notag
\intertext{or, equivalently}
1+\frac{\e^{\i \uplambda}}{b\cos\uplambda}\myb{\frac{zf'(z)}{pf(z)}-1} &\prec \frac{1+\myp{1-2\upalpha}z}{1-z} \Longrightarrow \myb{\frac{f(z)}{z^p}}^{a}\prec \frac{1}{\myp{1-z}^{2pab\myp{1-\upalpha}\e^{-\i \uplambda}\cos \uplambda}}.\notag
\end{align}
where  $2pab\myp{1-\upalpha}\e^{-\i \uplambda}\cos \uplambda \in 
\mathbf{B}$  and $\q(z)=\myp{1-z}^{-2p ab\myp{1-\upalpha}\e^{-\i 
\uplambda}\cos \uplambda}$ is the best dominant.
\end{corol}
For example, for $a=\frac12$ and  $p=b=1$ we get
\[
f\in \mathcal{S}^{\uplambda}(\upalpha)\Longrightarrow \sqrt{\frac{f(z)}{z}} \prec \frac{1}{\myp{1-z}^{\myp{1-\upalpha}\e^{-\i \uplambda}\cos \uplambda}}; \qquad \myp{z\in \mathbb{U}}.
\]
\begin{remark}
A special case of Corollary \ref{corol:th1:3} when $p=1$, $\upalpha=0$ 
and $f \in \mathcal{A}$ was given  by Aouf et al. {\cite[Theorem 
1]{AAH}}.
\end{remark}
Taking  $\upgamma=\frac{\e^{\i \uplambda}}{pab\cos\uplambda}$, $\upmu=0$, $\upeta=a$ and  $\q(z)=\myp{1-z}^{-2pab\myp{1-\upalpha}\e^{-\i \uplambda}\cos \uplambda}$ in Theorem \ref{th:a} and combining this together with item   \ref{UF1}, we obtain the following result:
\begin{corol}\label{corol:th1:4}
Let $f\in \mathcal{K}^{\uplambda}_p(\upalpha,b)$. Then
\begin{equation*}
\myb{\frac{f'(z)}{pz^{p-1}}}^{a} \prec \frac{1}{\myp{1-z}^{2pab\myp{1-\upalpha}\e^{-\i \uplambda}\cos \uplambda}}; \qquad \myp{z\in \mathbb{U}},
\end{equation*}
or, equivalently
\begin{align*}
1+\frac{\e^{\i \uplambda}}{b\cos\uplambda}\myb{\frac{1}{p}\myp{1+\frac{zf''(z)}{f'(z)}}-1}& \prec \frac{1+\myp{1-2\upalpha}z}{1-z}\notag\\
 \Longrightarrow \myb{\frac{f'(z)}{pz^{p-1}}}^{a}&\prec\frac{1}{\myp{1-z}^{2pab\myp{1-\upalpha}\e^{-\i \uplambda}\cos \uplambda}},
\end{align*}
where  $2p ab\myp{1-\upalpha}\e^{-\i \uplambda}\cos \uplambda\in 
\mathbf{B}$  and $\q(z)=\myp{1-z}^{-2pab\myp{1-\upalpha}\e^{-\i 
\uplambda}\cos \uplambda}$ is the best dominant.
\end{corol}
For example,  for $a=\frac12$ and $p=b=1$  we get
\[
f\in \mathcal{K}^{\uplambda}(\upalpha)\Longrightarrow \sqrt{f'(z)} \prec \frac{1}{\myp{1-z}^{\myp{1-\upalpha}\e^{-\i \uplambda}\cos \uplambda}}; \qquad \myp{z\in \mathbb{U}}.
\]
\begin{remark}
A special case of Corollary \ref{corol:th1:4} when $p=1$, $\upalpha=0$ 
and $f \in \mathcal{A}$ was given  by Aouf et al. {\cite[Corollary 
1]{AAH}}.
\end{remark}
\begin{theorem}\label{th:3}
Let $\q$ be univalent in $\mathbb{U}$ with $\q(0)=1$. Further, assume that $f \in \mathcal{A}_0$ and $\q$ satisfy the condition
\begin{equation}\label{th3:aaa}
\rs{1+\frac{z\q''(z)}{\q'(z)}}>\Max\set{0, -\RE\myp{\frac{\upsigma}{\upgamma}}}; \qquad \myp{z\in \mathbb{U}}.
\end{equation}
If the function $\uppsi$ define by
\begin{equation}\label{psi}
\Uppsi(z)\coloneqq \myb{\frac{f'(z)}{pz^{p-1}}}^{\upeta} 
\myb{\frac{z^p}{f(z)}}^{\upmu}\set{\upsigma+\upgamma
    \myb{\upeta\myp{1-p+\frac{zf''(z)}{f'(z)}}+\upmu
        \myp{p-\frac{zf'(z)}{f(z)}}}},
\end{equation}
satisfies the following subordination condition:
\begin{equation}\label{th3:aaaa}
\Uppsi(z) \prec \upsigma \q(z)+\upgamma z\q'(z); \qquad \myp{z\in \mathbb{U}}.
\end{equation}
  Then
\begin{equation*}
\mathscr{F}_{p}^{\upeta,\upmu}[f](z)\prec \q(z); \qquad \myp{z\in \mathbb{U}}.
\end{equation*}
and $\q$ is the best dominant of Eq.~\eqref{th3:aaaa}.
\end{theorem}
\begin{proof}
 If we choose $\uptheta(w)=\upsigma w$ and $\upvarphi(w)=\upgamma$, then $\uptheta, \upvarphi \in \mathcal{H}(\Omega)$; $\myp{\Omega\coloneqq \mathbb{C}}$. Also, for all $w\in \q(\mathbb{U})$, $\upvarphi(w)\neq 0$.   Define
 \begin{align*}
 Q(z)\coloneqq z \q'(z)\upvarphi(\q(z))=\upgamma z\q'(z),
 \end{align*}
The function  $\q$ is an univalent, then  $\q'(z)\neq0$ for all $z \in\mathbb{U}$,
$Q(0)=0$ and $Q'(0)=\upgamma \q'(0)\neq 0$, and from condition \eqref{th3:aaa}
 \begin{align*}
\rs{\frac{zQ'(z)}{Q(z)}}=\rs{1+\frac{z\q''(z)}{\q'(z)}}>0; \qquad \myp{z\in \mathbb{U}}.
 \end{align*}
Thus $Q$ is a starlike function in $\mathbb{U}$. Moreover, if
 \begin{align*}
 &h(z)\coloneqq\uptheta(\q(z))+Q(z)=\upsigma \q(z)+\upgamma z \q'(z),
 \intertext{then from condition \eqref{th3:aaa}, we  deduce}
& \rs{\frac{zh'(z)}{Q(z)}}=\rs{1+\frac{z\q''(z)}{\q'(z)}+\frac{\upsigma}{\upgamma}}>0; \qquad \myp{z\in \mathbb{U}}.
 \end{align*}
For $f\in \mathcal{A}_0$, the function $\mathscr{F}_{p}^{\upeta,\upmu}[f](z)$ given by \eqref{diff}, we have $\mathscr{F}_{p}^{\upeta,\upmu}[f](\mathbb{U})\subset\Omega$ and the subordinations \eqref{eq8} and \eqref{th3:aaaa} are equivalent, then  all the conditions of Lemma \ref{lem1} are satisfied and the function $\q$ is the best dominant of \eqref{th:aaaa}.
\end{proof}
Taking   $q(z)=\myp{1+Az}/\myp{1+Bz}$; $\left(-1\leq B<A\leq1,\: z\in \mathbb{U}\right)$ in Theorem \ref{th:3} and then applying Lemma \ref{lemma4},  we obtain the following result:
\begin{corol}
Let $-1\leq B<A\leq1$ and 
\[
\rs{\frac{\upsigma}{\upgamma}}\geq \frac{|B|-1}{|B|+1}.
\]
  If $f\in \mathcal{A}_0$ and the  function $\Uppsi$ given by \eqref{psi} satisfies the subordination
\begin{align}
&\Uppsi(z) \prec \upsigma\myp{\frac{1+Az}{1+Bz}}+ \frac{\upgamma\myp{A-B} z}{\myp{1+Bz}^2}; \qquad \myp{z\in \mathbb{U}},\label{eqcor1th31}
\intertext{then}
 &\mathscr{F}_{p}^{\upeta,\upmu}[f](z) \prec \frac{1+Az}{1+Bz}; \qquad \myp{z\in \mathbb{U}}.\notag
\end{align}
 and $\myp{1+Az}/\myp{1+Bz}$  is the best dominant of Eq.~\eqref{eqcor1th31}.
\end{corol}
For $\q(z)=\e^{Cz}$; $\myp{|C|<\uppi}$ in Theorem \ref{th:3}, we obtain the following corollary.
\begin{corol}
Let 
\[
\rs{\frac{\upsigma}{\upgamma}}\geq |C|-1; \qquad \myp{|C|<\uppi}.
\]
If $f\in \mathcal{A}_0$ and the  function $\Uppsi$ given by \eqref{psi} satisfies the subordination
\begin{align}
&\Uppsi(z) \prec \myp{\upsigma+\upgamma C z}\e^{Cz}; \qquad \myp{z\in \mathbb{U}},\label{eqcor1th32}
\intertext{then}
 &\mathscr{F}_{p}^{\upeta,\upmu}[f](z) \prec \e^{Cz}; \qquad \myp{z\in \mathbb{U}},\notag
\end{align}
 and $\e^{Cz}$  is the best dominant of Eq.~\eqref{eqcor1th32}.
\end{corol}
Taking   $q(z)=\myp{1+Az}/\myp{1+Bz}$; $\left(-1< B<A\leq1,\: z\in \mathbb{U}\right)$ in Theorem \ref{th:3},  we obtain the following result:
\begin{corol}
Let $-1< B<A\leq1$ and $\rs{u-vB}\geq\abs{v-uB}$ where $u=1+\frac{\upsigma}{\upgamma}$ and $v=\frac{B\myp{\upsigma-\upgamma}}{\upgamma}$. If $f\in \mathcal{A}_0$ and the  function $\Uppsi$ given by \eqref{psi} satisfies the subordination
\begin{align}
&\Uppsi(z) \prec \upsigma\myp{\frac{1+Az}{1+Bz}}+ \frac{\myp{A-B}\upgamma z}{\myp{1+Bz}^2}; \qquad \myp{z\in \mathbb{U}},\label{eqcor1th33}
\intertext{then}
 &\mathscr{F}_{p}^{\upeta,\upmu}[f](z) \prec \frac{1+Az}{1+Bz}; \qquad \myp{z\in \mathbb{U}},\notag
\end{align}
 and $\myp{1+Az}/\myp{1+Bz}$  is the best dominant of Eq.~\eqref{eqcor1th33}.
\end{corol}
\begin{proof}
Let $\q(z)=\myp{1+Az}/\myp{1+Bz}$, then we have
\[
z\q'(z)=\frac{\myp{A-B}z}{\myp{1+Bz}^2}\qquad \text{and}\qquad 1+\frac{z\q''(z)}{\q'(z)}=\frac{1-Bz}{1+Bz}.
\]
Thus
\[
\frac{\upsigma}{\upgamma}+1+\frac{z\q''(z)}{\q'(z)}=\frac{u+vz}{1+Bz},
\]
where $u=1+\frac{\upsigma}{\upgamma}$ and 
$v=\frac{B\myp{\upsigma-\upgamma}}{\upgamma}$. According to Lemma 
\ref{lemma3}, it follows that
\begin{equation*}
\rs{\frac{\upsigma}{\upgamma}+1+\frac{z\q''(z)}{\q'(z)}}>\frac{\rs{u-vB}-\abs{v-uB}}{1-B^2}\geq 0.
\end{equation*}
By using Theorem \ref{th:3}, we obtain the required result.
\end{proof}
\section{Superordination Results}\label{sec:sup}
\begin{theorem}\label{th:2}
Let $\q$ be  convex  function  in $\mathbb{U}$ with $\q(0)=1$. Further, assume that $\rs{\frac{\upsigma}{\upgamma}}>0$ and the functions  $f \in \mathcal{A}_0$ and $\q$ satisfy the conditions
\begin{equation*}
\mathscr{F}_{p}^{\upeta,\upmu}[f](z)\in\mathcal{H}[\q(0),1]\cap \mathcal{Q}; \qquad \myp{z\in \mathbb{U}}.
\end{equation*}
If the function $\Uppsi$ given by \eqref{psi}
is univalent in $\mathbb{U}$, and satisfies the following subordination condition:
\begin{equation}\label{th2:aaaa}
\upsigma \q(z)+\upgamma z\q'(z) \prec \Uppsi(z); \qquad \myp{z\in \mathbb{U}}.
\end{equation}
  Then
\begin{equation*}
\q(z)\prec \mathscr{F}_{p}^{\upeta,\upmu}[f](z); \qquad \myp{z\in \mathbb{U}},
\end{equation*}
and $\q$ is the best subordinant of Eq.~\eqref{th2:aaaa}.
\end{theorem}
\begin{proof}
Let $f\in \mathcal{A}_0$. Define the function $g$ by
\[
g(z)\coloneqq \mathscr{F}_{p}^{\upeta,\upmu}[f](z) =\myb{\frac{f'(z)}{pz^{p-1}}}^{\upeta} \myb{\frac{z^p}{f(z)}}^{\upmu}; \qquad \myp{z\in \mathbb{U}}.
\]
 Differentiating $g(z)$ logarithmically with respect to $z$, we get
\[
\frac{zg'(z)}{g(z)}=\upeta\myp{1-p+\frac{zf''(z)}{f'(z)}}+\upmu\myp{p-\frac{zf'(z)}{f(z)}}; \qquad \myp{z\in \mathbb{U}},
\]
hence the subordination \eqref{th2:aaaa} is equivalent to
\[
\upsigma \q(z)+\upgamma z\q'(z) \prec \upsigma g(z)+\upgamma zg'(z).
\]
By using Lemma \ref{SRS}, we obtain the required result.
\end{proof}
Taking  $\upeta=1$ and $\upmu=0$ in Theorem \ref{th:2},  we obtain the following result:
\begin{corol}
Let $\q$ be  convex  function in $\mathbb{U}$ with $\q(0)=1$. Further, assume that the functions  $f \in \mathcal{A}_p$ and $\q$ satisfy the conditions
\begin{align*}
\frac{f'(z)}{pz^{p-1}}\in\mathcal{H}[\q(0),1]\cap \mathcal{Q}; \qquad \myp{z\in \mathbb{U}}.
\end{align*}
If the function
\[
\phi(z)\coloneqq \frac{f'(z)}{pz^{p-1}}\myb{2-p+\frac{zf''(z)}{f'(z)}}=\myb{\frac{zf'(z)}{pz^{p-1}}}',
\]
 is univalent in $\mathbb{U}$, and satisfies the following subordination condition:
\begin{align}
 \myb{z\q(z)}'  &\prec\myb{\frac{zf'(z)}{pz^{p-1}}}'; \qquad \myp{z\in \mathbb{U}}.\label{corollary:th2:aaaa}
\intertext{Then}
\q(z)  & \prec \frac{f'(z)}{pz^{p-1}}; \qquad \myp{z\in \mathbb{U}},\notag
\end{align}
and $\q$ is the best subordinant of \eqref{corollary:th2:aaaa}.
\end{corol}
Taking  $\upmu=\upeta=1$ in Theorem \ref{th:2},  we obtain the following result:
\begin{corol}
Let $\q$ be  convex  function in $\mathbb{U}$ with $\q(0)=1$. Further, assume that  the functions  $f \in \mathcal{A}_p$ and $\q$ satisfy the conditions
\begin{align*}
\frac{1}{p}\frac{zf'(z)}{f(z)}\in\mathcal{H}[\q(0),1]\cap \mathcal{Q}; \qquad \myp{z\in \mathbb{U}}.
\end{align*}
If the function
\[
\uppsi(z)\coloneqq \frac{1}{p}\myb{2+\frac{zf''(z)}{f'(z)}-\frac{zf'(z)}{f(z)}}\frac{zf'(z)}{f(z)}=\myb{\frac{1}{p}\frac{z^2f'(z)}{f(z)}}',
\]
 is univalent in $\mathbb{U}$, and satisfies the following subordination condition:
\begin{align}
 \myb{z\q(z)}'&\prec \myb{\frac{1}{p}\frac{z^2f'(z)}{f(z)}}'; \qquad \myp{z\in \mathbb{U}}.\label{corollary1:th2:aaaa}
\intertext{Then}
\q(z)  & \prec \frac{1}{p}\frac{zf'(z)}{f(z)}; \qquad \myp{z\in \mathbb{U}},\notag
\end{align}
and $\q$ is the best subordinant of Eq.~\eqref{corollary1:th2:aaaa}.
\end{corol}\medskip
Combining Theorem \ref{th:3} with Theorem \ref{th:2}, we obtain the following ``sandwich result''.
\begin{theorem}
Let $\q_1$   and $\q_2$ be  convex and convex (univalent) functions in $\mathbb{U}$ with $\q_1(0)=\q_2(0)=1$ respectively. Further, assume that $\rs{\frac{\upsigma}{\upgamma}}>0$ and  function  $f \in \mathcal{A}_0$ satisfy the condition
\begin{equation*}
\mathscr{F}_{p}^{\upeta,\upmu}[f](z)\in\mathcal{H}[1,1]\cap \mathcal{Q}; \qquad \myp{z\in \mathbb{U}}.
\end{equation*}
If the function $\Uppsi$ given by \eqref{psi}
is univalent in $\mathbb{U}$, and satisfies the following subordination condition:
\begin{equation}\label{th4:aaaa}
\upsigma \q_1(z)+\upgamma z\q'_1(z) \prec \Uppsi(z)\prec \upsigma \q_2(z)+\upgamma z\q'_2(z); \qquad \myp{z\in \mathbb{U}}.
\end{equation}
  Then
\begin{equation*}
\q_1(z)\prec \mathscr{F}_{p}^{\upeta,\upmu}[f](z)\prec \q_2(z); \qquad \myp{z\in \mathbb{U}},
\end{equation*}
and $\q_1$ and $\q_2$ are respectively the best subordinant and  best 
dominant of Eq.~\eqref{th4:aaaa}.
\end{theorem}


\begin{thebibliography}{9}
\bibitem{AAH}
M. K. Aouf, F. M. Al-Oboudi and M. M. Haidan,
On some results for $\uplambda$-spirallike and $\uplambda$-Robertson functions of complex order, \textit{Publ. Inst. Math.} \textbf{77} (2005), no. 91, 93--98.

\bibitem{AH}
F. M. Al-Oboudi and M. M. Haidan,
Spirallike functions of complex order,
\textit{J. Natural Geom.} \textbf{19} (2000), 53--72.

\bibitem{AMZ}
M. K. Aouf, A. O. Mostafa and H. M. Zayed,
Subordination and superordination properties of $p$-valent functions defined by a generalized fractional differintegral operator,
\textit{Quaest. Math.} \textbf{39} (2016), no. 4, 545--560.

 \bibitem{Bul}
T. Bulboac\u{a},
Differential Subordinations and Superordinations: Recent Results,
 \textit{Casa C{\u{a}}r{\c{t}}ii de {\c{S}}tiin{\c{t}}{\u{a}}} (2005).

\bibitem{bul2}
T. Bulboac\u{a},
Classes of first-order differential superordinations,
\textit{Demonstratio Math.} \textbf{35} (2002), no.2,  287--292.

\bibitem{BBS}
S. Z. H. Bukhari, T. Bulboac\u{a} and  M. S.  Shabbir,
Subordination and superordination results for analytic functions with respect to symmetrical points, \textit{Quaest. Math.} \textbf{41} (2018), no. 1, 65--79.

\bibitem{BT}
T. Bulboac{\u{a}} and N. Tuneski, 
Sufficient conditions for bounded turning of analytic functions,
\textit{Ukrains’ kyi Matematychnyi Zhurnal}  \textbf{70} (2018), no. 08,  1118--1127.

\bibitem{EMN}
A. Ebadian, V. S. Masih and Sh. Najafzadeh, 
Some extension results concerning analytic and meromorphic multivalent functions
\textit{Bull. Korean Math. Soc.} (Accepted).

\bibitem{ES}
A. Ebadian and J. Sok{\'o}{\l},
\textit{On the subordination and superordination of strongly starlike functions},
Math. Slovaca, \textbf{66} (2016), 815--822.

\bibitem{MM1}
S. S. Miller and P. T.  Mocanu,
Differential Subordinations: Theory and Applications,
 Series on monographs and textbooks in pure and appl. math., vol. 255. Marcel Dekker, Inc., New York (2000).

\bibitem{MM2}
S. S. Miller and P. T.  Mocanu,
On some classes of first-order differential subordinations,
\textit{Michigan Math. J.} \textbf{32} (1985), 185--195.

\bibitem{MM3}
S. S. Miller and P. T.  Mocanu,
Differential subordinations and univalent functions,
\textit{Michigan Math. J.} \textbf{28} (1981), no. 2, 157--172.

\bibitem{MM4}
S. S. Miller and P. T.  Mocanu,
Subordinants of differential superordinations,
\textit{Complex variables} \textbf{48} (2003), no. 10, 815--826.

\bibitem{Roy}
W. C. Royster, 
On the univalence of a certain integral,
\textit{Michigan Math. J.} \textbf (1965), no. 4,  385--387.

\bibitem{RR}
R. M. Ali and V. Ravichandran,
Integral operators on Ma--Minda type starlike and convex functions,
\textit{Math. Comput. Model.} \textbf{53} (2011), no. 5-6, 581--586.

\bibitem{SRMS}
 N. Shammugam, C. D. Ramachandran,  M. Darus and S. Sivasubramanian, Differential sandwich theorems for some subclasses of analytic functions involving a linear operator,
\textit{Acta Math. Univ. Comenianae} \textbf{76} (2007), 287--294.

\bibitem{SRS}
T. N. Shanmugam, V. Ravichandran and S. Sivasubramanian,
Differential sandwich theorems for some subclasses of analytic functions,
 \textit{Austral. J. Math. Anal. Appl.} \textbf{3} (2006), no. 1, 1--11.

\bibitem{ZMA}
H. M. Zayed, S. A. Mohammadein and  M. K. Aouf,
Sandwich results of $p$-valent functions defined by a generalized fractional derivative operator with application to vortex motion,
\textit{Revista de la Real Academia de Ciencias Exactas, Fisicas y Naturales. Serie A. Matem{\'a}ticas}  (2018), 1--16.

\end{thebibliography}
\end{document}